\documentclass[12pt]{article}
\usepackage[english]{babel}

\usepackage{amsmath}
\usepackage{amssymb,amsthm}

\usepackage{mathrsfs}

\usepackage{bbm}

\numberwithin{equation}{section}

\newtheorem{theorem}{Theorem}

\newtheorem{lemma}{Lemma}[section]

\newtheorem{remark}[lemma]{Remark}

\newtheorem{notation}[lemma]{Notation}
\newtheorem{definition}[lemma]{Definition}

\renewcommand{\leq}{\leqslant}
\renewcommand{\geq}{\geqslant}

\newcommand{\Rot}{\mathcal{R}}

\makeatletter
\newcommand*{\IfItalic}{%
  \ifx\f@shape\my@test@it
    \expandafter\@firstoftwo
  \else
    \expandafter\@secondoftwo
  \fi
}
\newcommand*{\my@test@it}{it}
\makeatother

\newcommand{\myae}{\IfItalic{\emph{\mbox{\ae}}}{\mbox{\ae}}}

\newcommand\lI{L}

\newcommand\ssemg\psi
\newcommand\fkinkg\gamma
\newcommand\iks[1]{\widehat{x}_{#1}}
\newcommand\iksp[1]{\widehat{x}_{#1}^{+}}
\newcommand\ikspm[1]{\widehat{x}_{#1}^{\, \pm}}

\usepackage[all]{xy}

\usepackage{amscd}

\begin{document}

\begin{center}
{\Large On one generalization of skew tent maps}

{\large Makar Plakhotnyk\\
University of S\~ao Paulo, Brazil~\footnote{This work is partially
supported by FAPESP (S\~ao Paulo,
Brazil).}.\\

makar.plakhotnyk@gmail.com}\\
\end{center}

\begin{abstract}
We generalize in this work the properties of the conjugacy of skew
tent maps. It is known that the conjugacy $h$ from a skew tent map
$g_1$ to $g_2$ is differentiable at a point $x^*$ if and only if
there exists left and right limits $\lim\limits_{n\rightarrow
\infty}h_n'(x^*-)$ and $\lim\limits_{n\rightarrow
\infty}h_n'(x^*+)$, where $h_n$ is a piecewise linear function,
which coincides with $h$ at $g_1^{-n}(0)$, and all whose kinks
belong to $g_1^{-n}(0)$. The attempts to generalize this result to
some reacher class of unimodal maps is natural. For this reason we
introduce the class of piecewise linear maps, all whose kinks are
in the complete pre-image of $0$ and study the relation of the
differential properties of their conjugacy with ones of the
mentioned approximation~$(h_n)_{n\geq 1}$.~
\end{abstract}

\section{Introduction}


The final aim of the theory of the motions of dynamical systems
must be directed toward the qualitative determination of all
possible types of motions and of the interrelation of these
motions~\cite[p. 189]{Birk-1927}. Topological conjugation is the
classical tool to divide the dynamical systems to classes of
equivalence, where all the trajectories are the same in a certain
cense. S. Ulam obtained in~\cite{Ulam-1964-b} the topological
conjugacy $x\mapsto \frac{2}{\pi}\arcsin\sqrt{x}$ of the
\textbf{tent} map $f:\, x\mapsto 1-|1-2x|$ and the logistic map
$x\mapsto 4x(1-x)$, which are defined on $x\in [0, 1]$. Due to the
simplicity and intuitive unexpectedness of this result, it rapidly
became a part of the scientific folklore and the inevitable
example of any textbook on the Dynamical Systems theory.

We will call a continuous map $g: [0, 1]\rightarrow [0, 1]$
\textbf{unimodal}, if it can be written in the form
\begin{equation}\label{eq:1.1} g(x) = \left\{
\begin{array}{ll}g_l(x),& 0\leq x\leq
v,\\
g_r(x), & v\leq x\leq 1,
\end{array}\right.
\end{equation} where %
$v\in (0,\, 1)$ is a parameter, the function $g_l$ increase, the
function $g_r$ decrease, and
$$g(0)=g(1)=1-g(v)=0.$$

If a homeomorphism $h$ satisfies the functional equation $ h\circ
g_1 = g_2\circ h,$ where $g_1$ and $g_2$ are unimodal maps of the
form~\eqref{eq:1.1}, then we will say that $h$ is a
\textbf{conjugation from $g_1$ to~$g_2$}.

\begin{theorem}\label{th:1}\cite[p. 53]{Ulam-1964-b} %
A unimodal map $g$ is topologically conjugated to the tent map if
and only if the complete pre-image of $0$ under the action of $g$
is dense in $[0,\, 1]$.
\end{theorem}

Recall, that the set $g^{-\infty}(a) = \bigcup\limits_{n\geq
1}g^{-n}(a)$, where $g^{-n}(a) = \left\{ x\in [0,\, 1]:\, g^n(x) =
a\right\}$ for all $n\geq 1$, is called the complete pre-image of
$a$ (under the action of the map $g$). We have made
in~\cite{Chaos} the following construction, being motivated by
Theorem~\ref{th:1}.

\begin{remark}\cite[Remark~2.1]{Chaos}\label{rem:1.1}
For any $n\geq 1$ the set $g^{-n}(0)$ consists of $2^{n-1}+1$
points.
\end{remark}

Due to Remark~\ref{rem:1.1}, the notation appears.

\begin{notation}\cite[Notation~2.1]{Chaos}
For every unimodal map $g:\, [0,\, 1]\rightarrow [0,\, 1]$ and for
every $n\geq 1$ denote $\{ \mu_{n,k}(g),\, 0\leq k\leq 2^{n-1}\}$
such that $g^n(\mu_{n,k}(g))=0$ and $\mu_{n,k}(g)<\mu_{n,k+1}(g)$
for all~$k$.
\end{notation}

\begin{lemma}\cite[Lemma 3]{Chaos}\label{lema:1.3}
Suppose that $g_1,\, g_2:\, [0,\, 1]\rightarrow [0,\, 1]$ are
unimodal maps, and $h$ is the conjugacy from $g_1$ to $g_2$. Then
$h(\mu_{n,k}(g_1)) =\mu_{n,k}(g_2)$ for all $n\geq 1$ and $k,\,
0\leq k\leq 2^{n-1}$.
\end{lemma}

Lemma~\ref{lema:1.3} motivates the next construction.

\begin{definition}
For %
each $n\geq 1$ denote by $h_n:\, [0, 1]\rightarrow [0, 1]$ the
piecewise linear map, such that $h_n(\mu_{n,k}(g_1)) =
\mu_{n,k}(g_2)$ for all $k,\, 0\leq k\leq 2^{n-1}$, and $h_n$ is
linear at all other points. We call the sequence $\{h_n, n\geq
1\}$ the \textbf{Ulam's approximation} of the conjugacy from $g_1$
to $g_2$.\end{definition}

\begin{definition}
For the Ulam's approximation and $x<1$ denote $L_h(x)
=\lim\limits_{n\rightarrow \infty}h_n'(x-)$, and for every $x>0$
write $R_h(x) =\lim\limits_{n\rightarrow
\infty}h_n'(x+)$.~\footnote{we denote by $h_n'(x-)$ the left
derivative of $h_n$ at $x$. Analogously, $h_n'(x+)$ means the
right derivative.} Write for convenience  $L_h(0) = R_h(0)$ and
$R_h(1) = L_h(1)$.
\end{definition}

The simplest and natural example of a unimodal map is the case,
when graphs of both $g_l$ and $g_r$ in~\eqref{eq:1.1} are line
segments, i.e. $g$ is of the form \begin{equation}\label{eq:1.2}
f_v(x) = \left\{\begin{array}{ll}
\frac{x}{v},& \text{if }0\leq x\leqslant v,\\
 \frac{1-x}{1-v},& \text{if }
v\leq x\leq 1.
\end{array}\right.
\end{equation}

The topological conjugation of maps of the form~\eqref{eq:1.2} are
studied since early 1980-th, see~\cite{Jydy}, \cite{Skufca}
and~\cite{Yong-Guo-Wang}. We will call a function of the
form~\eqref{eq:1.2} \textbf{skew tent map} due to~\cite{Skufca}
and~\cite{Yong-Guo-Wang}, where Ulam's approximations of
conjugacies of skew tent maps were studied. It is known
(see~\cite{Jydy}) that all skew tent maps are topologically
conjugated to the tent. We have proved in~\cite{Chaos} the next
fact about the conjugacy of skew tent maps.

\begin{theorem}\label{th:2}\cite[Theorem~2]{Chaos}
For every $v_1,\, v_2\in (0,\, 1)$ let $f_{v_1}$ and $f_{v_2}$ be
skew tent maps, and let $h$ be the conjugacy from $f_{v_1}$ to
$f_{v_2}$. Let for any $x\in [0, 1]$ the limits $L(x)$ and $R(x)$
are defined by Ulam's approximation of $h$. Then the derivatives
$h'(x-)$ and $h'(x+)$ exist if and only if there exist $L(x)$ and
$R(x)$ respectively. Moreover, in these cases $h'(x-) = L(x)$ and
$h'(x+) = R(x)$.
\end{theorem}

The aim of this article is to generalize Theorem~\ref{th:2} to a
reacher class of maps. We define a piecewise linear function $g$
of the form~\eqref{eq:1.1} to be a \textbf{carcass map}, if every
its kink belongs to $g^{-\infty}(0)$. We prove the following facts
about carcass maps:

\begin{theorem}\label{th:3}
For every carcass map $g$ the set $g^{-\infty}(0)$ is dense in
$[0, 1]$.
\end{theorem}

Theorems~\ref{th:1} and~\ref{th:3} yield that every carcass maps
$g_1$ and $g_2$ are topologically conjugated.

\begin{theorem}\label{th:4}
Let $h$ be the conjugacy from a carcass map $g_1$ to a carcass map
$g_2$. Then the following implications hold:

(1) If for some $x\in [0, 1]$ the limits $L(x)$ and $R(x)$ exist,
are equal, and either $L(x)=0$ or $L(x)=\infty$, then the
derivative $h'(x)$ exists and $h'(x) =L(x)$.

(2) If for at least one $x\in [0,\, 1]$ the derivative $h'(x)$
exists, is positive and finite, then $h$ is piecewise linear.
\end{theorem}

We use Theorem~\ref{th:4} as a main tool in the proof of the next
fact.

\begin{theorem}\label{th:5}
There exists a carcass map $g$ and a point $x^*\in [0, 1]$ with
the following properties:

(1) The limits $L(x^*)$ and $R(x^*)$ exist, are positive and
finite, equal one to another;

(2) the derivative of $h$ does not exist at $x^*$.
\end{theorem}

Notice that the idea of the proof of Theorem~\ref{th:4} is similar
to one of the proof of Theorem~\ref{th:2} in~\cite{Chaos}.
Nevertheless, it is more interesting, that Theorem~\ref{th:2} can
not be generalized to carcass maps.

\section{Preliminaries}\label{sec:02}

Let a number $x\in [0, 1]$ and a unimodal map $g$ be fixed till
the end of this section. Write, for simplicity, $\mu_{n,k}$
instead of $\mu_{n,k}(g)$ for all $n\geq 1$ and $k\in\{0,\ldots,
2^{n-1}\}$.

\begin{lemma}\cite[Lemma 1]{Chaos}\label{lema:2.1}
For every $n\geq 2$ and $k,\, 0\leq k\leq 2^{n-2}$, we have that:
$$\begin{array}{ll}
(i)\ g(\, \mu_{n,k}\, )=\mu_{n-1,k}; & %
\hskip 1 cm(ii)\ g(\, \mu_{n,k}\, )= g( \mu_{n,2^{n-1}-k} ).
\end{array}$$

\end{lemma}

\begin{notation}\cite[Notation~4.4]{Chaos}\label{not:2.2}
1. For every $n\geq 1$ denote by $k_n$ the maximal $k\in \{
0,\ldots, 2^{n-1}-1\}$ such that $x\in [\mu_{n,k}, \mu_{n,k+1}]$.

2. For every $n\geq 0$ denote $\iks{n} =\mu_{n+1,k_n}, $ $\iksp{n}
= \mu_{n+1,k_n+1},$ $\ikspm{n} = \mu_{n+2,2k_n+1}, $ $
\widehat{x}_n^{\, -} =\mu_{n+1,k_n-1}$ and $\widehat{x}_n^{\, \mp}
=\mu_{n+2,2k_n-1}$.

3. For any $n\geq 1$ and $k,\, 0\leq k< 2^{n-1}$ denote $I_{n,k} =
(\mu_{n,k},\, \mu_{n,k+1})$ and $\lI_{n,k} =
\mu_{n,k+1}-\mu_{n,k}$.

4. Denote $I_n(x) = I_{n,k_{n-1}}$ and $I_n^-(x) =
I_{n,k_{n-1}-1}$. Also $\lI_n(x) =\lI_{n,k_{n-1}}$ and $\lI_n^-(x)
=\lI_{n,k_{n-1}-1}$.
\end{notation}

\begin{lemma}\cite[Lemma~7]{Chaos}
For every $x\in [0, 1]$ there exists a sequence $(x_i)_{i\geq 0}$
of $x_i\in \{0; 1\}$ with $x_0=0$, such that $k_n =
\sum\limits_{i=1}^{n}x_i2^{n-i}$ for all $n\geq 1$.
\end{lemma}

We will call the sequence $(x_i)_{i\geq 0}$, constructed by $x\in
[0, 1]$ and unimodal map $g$, the $g$-\textbf{expansion} of $x$
(with respect to $g$). Say that $x$ is $g$-\textbf{finite}, if
there is $k$ such that $x_i=0$ for all $i>k$. Otherwise say that
$x$ is $g$-\textbf{infinite}. The next fact follows from the
construction.

\begin{lemma}\label{lema:2.4}
If $h$ is the conjugacy from a unimodal map $g_1$ to unimodal map
$g_2$, then for every $x\in [0,\, 1]$ the $g_2$-expansion of
$h(x)$ coincides with the $g_1$-expansion of $x$~\footnote{This
lemma for skew tent maps was proved in~\cite[Lemma~16]{Chaos}}.
\end{lemma}

\begin{notation}
For every $i,j:\, 0\leq i<j$ denote $k_{i,j} =
\sum\limits_{t=i}^{j}x_t2^{j-t}$. Remark that $k_n = k_{0,n} =
k_{1,n}$.
\end{notation}

\begin{remark}\label{rem:2.6}
Notice that $I_{n+1,k_n} = (\iks{n}, \iksp{n})$ and $I_{n+2,2k_n}
= (\iks{n}, \ikspm{n})$.
\end{remark}

\begin{notation}\label{not:2.7}
For any $n\geq 1$ and $k,\, 0\leq k< 2^n$ denote
$$\delta_{n,k} =
\frac{\mu_{n+1,2k+1}-\mu_{n,k}}{\mu_{n,k+1}-\mu_{n,k}} =
\frac{\lI_{n+1,2k}}{\lI_{n,k}}.$$ If $I$ is an interval of the
form $I = I_{n,k}$ then write $\delta(I)$ for $\delta_{n,k}$. This
will be convenient, for example, in the expression
$\delta(g(I_{n,k}))$.
\end{notation}

The next fact follows from the definitions.

\begin{remark}\label{rem:2.8}
For every $n\geq 0$ we have  $ \ikspm{n} = \iks{n}
+\delta_{n+1,k_n}\cdot \lI_{n+1,k_n}.$
\end{remark}

\begin{remark}\cite[Remark~4.6]{Chaos}\label{rem:2.9}
Let $n\geq 0$ and $x_{n+1}=0$. Then:
$$\begin{array}{ll}
(i)\ \iks{n+1}  =\iks{n}; & %
\hskip 1 cm(ii)\ \iksp{n+1} = \ikspm{n+1}.
\end{array}$$
\end{remark}

\begin{remark}\cite[Remark~4.7]{Chaos}\label{rem:2.10}
Let $n\geq 0$ and $x_{n+1}=1$. Then:
$$\begin{array}{lll}
(i)\ \widehat{x}_{n+1} =\widehat{x}_n^{\, \pm}; & %
\hskip 1 cm(ii)\ \widehat{x}_{n+1}^{\, +} =\widehat{x}_n^{\, +}.
\end{array}$$
\end{remark}

\begin{lemma}\label{lema:2.11}
For any $n\geq 1$ we have that:

(i) $\iks{n+1} =\widehat{x}_n + x_{n+1}\cdot \delta_{n+1,k_n}\cdot
\lI_{n+1,k_n}$

(ii) $\iksp{n+1} =\iksp{n} - (1-x_{n+1})\cdot
(1-\delta_{n+1,k_n})\cdot \lI_{n+1,k_n}$;

\end{lemma}

\begin{proof}
If $x_{n+1} =0$, then~(i) follows from~(i) of
Remark~\ref{rem:2.9}. If $x_{n+1} =1$ then, by~(i) of
Remark~\ref{rem:2.10}, we have $\iks{n+1} = \ikspm{n}$, and part
(i) follows from Remark~\ref{rem:2.8}.

If $x_{n+1} =1$, then~(ii) follows from (ii) of
Remark~\ref{rem:2.10}. If $x_{n+1} =0$, then, by~(i) of
Remark~\ref{rem:2.9}, $\iksp{n+1} = \ikspm{n}$. Now, by
Remark~\ref{rem:2.8}, $$ \ikspm{n} =\iks{n} +\delta_{n+1,k_n}\cdot
\lI_{n+1,k_n} \stackrel{\text{by Rem.~\ref{rem:2.6}}}{=} \iksp{n}
-\lI_{n+1,k_n} +\delta_{n+1,k_n}\cdot \lI_{n+1,k_n},
$$ which is necessary.
\end{proof}

\begin{notation}
1. For every $t\in [0, 1]$ denote $\Rot(t) = 1-t$.

2. For every $n\geq 1$ and $t\in \{0,\ldots, 2^{n}-1\}$ denote
$\Rot_n(t) = 2^n-t-1$.
\end{notation}

\begin{remark}\label{rem:2.13}
If $t = \sum\limits_{i=1}^{n}t_i\cdot 2^{n-i}$ is the binomial
expansion of $t$, then $\Rot_n(t) =
\sum\limits_{i=1}^{n}\Rot(t_i)\cdot 2^{n-i}$. In other words,
$\Rot_n$ inverts the binomial digits of a number, whose binomial
expansion consists of $n$ digits.
\end{remark}

\begin{proof}
Notice that $\sum\limits_{i=1}^{n}t_i\cdot 2^{n-i} +
\sum\limits_{i=1}^{n}\Rot(t_i)\cdot 2^{n-i} =
\sum\limits_{i=1}^{n}2^{n-i} = 2^n-1$, and the fact follows.
\end{proof}

\begin{remark}\label{rem:2.14}
For any $x_1, x_2\in \{0; 1\}$ and $t\in [0, 1]$ we have that
$
\Rot^{\Rot^{x_1}(x_2)}(t) = \Rot^{x_1+x_2}(t),
$ where the power %
of $\Rot$ denotes its iteration.
\end{remark}

\begin{proof}
The proof follows immediately from the dichotomy whether $x_1
=x_2$, or not.
\end{proof}

\begin{remark}\label{rem:2.15}
For every $n\geq 1$ we have that $ \lI_{n+1,k_n} =
\lI_{n,k_{n-1}}\cdot \Rot^{x_n}(\delta_{n,k_{n-1}}).
$
\end{remark}

\begin{proof}
By Lemma~\ref{lema:2.11}, $$ \lI_{n+1,k_n} =\iksp{n} -\iks{n} =$$
$$= (\iksp{n-1} - (1-x_{n})\cdot (1-\delta_{n,k_{n-1}})\cdot
\lI_{n,k_{n-1}}) -(\iks{n-1} + x_{n}\cdot \delta_{n,k_{n-1}}\cdot
\lI_{n,k_{n-1}})=$$$$= (x_n -2\cdot x_n\cdot \delta_{n,k_{n-1}}
+\delta_{n,k_{n-1}})\cdot \lI_{n,k_{n-1}}.
$$ If $x_n=1$, %
then $\lI_{n+1,k_n} = (1 -\delta_{n,k_{n-1}})\cdot
\lI_{n,k_{n-1}}$; if $x_1=0$, then $\lI_{n+1,k_n} =
\delta_{n,k_{n-1}}\cdot \lI_{n,k_{n-1}}$, and we are done.
\end{proof}

\begin{remark}\label{rem:2.16}
For any $n\geq 1$ one have that $ g(I_{n+1,k_n}) =
I_{n,\Rot_{n-1}^{x_1}(k_{2,n})}. $
\end{remark}

\begin{proof}
If $x_1=0$, then $g$ increase on $I_{n+1,k_n}$, whence the left
point of $g(I_{n+1,k_n})$ is
$$g(\mu_{n+1,k_n})
\stackrel{\text{(i) of Lem.~\ref{lema:2.1}}}{=} \mu_{n-1,k_n}
\stackrel{x_1 =0}{=} \mu_{n-1,\Rot_{n-1}^{x_1}(k_{2,n})}.
$$

If $x_1=1$, then then $g$ decrease on $I_{n+1,k_n}$, and the left
point of $g(I_{n+1,k_n})$ is
$$g(\mu_{n+1,k_n+1})
\stackrel{\text{(ii) of Lem.~\ref{lema:2.1}}}{=}
g(\mu_{n+1,2^n-k-1}) \stackrel{\text{Rem.~\ref{rem:2.13}}}{=}
\mu_{n-1,\Rot_{n-1}(k_{2,n})}.
$$
\end{proof}

\begin{remark}\label{rem:2.17}
For any $n\geq 1$ and all $i, 1\leq i\leq n$, one have $
g^i(I_{n+1,k_n}) = I_{n+1-i,\Rot_{n-i}^{x_i} (k_{i+1, n})}. $
\end{remark}

\begin{proof}First write
$ I_{n+1,k_n} \stackrel{\text{Rem.~\ref{rem:2.16}}}{=}
I_{n,\Rot_{n-1}^{x_1}(k_{2,n})}$. Next,
$$
g(I_{n,\Rot_{n-1}^{x_1}(k_{2,n})})
\stackrel{\text{Rem.~\ref{rem:2.16}}}{=}
I_{n-1,\Rot^{\Rot^{x_1}(x_2)}(\Rot_{n-2}^{x_1}(k_{3,n}))}
\stackrel{\text{Rem.~\ref{rem:2.14}}}{=} $$
$$=I_{n-1,\Rot^{x_1 +x_2}\circ \Rot_{n-3}^{x_1}(k_{3,n})} =
I_{n-1,\Rot_{n-2}^{x_2}(k_{3,n})}
$$ and the result follows by induction on $i$.
\end{proof}

Denote by $n_0$ the minimal natural number such that $g^{-n_0}(0)$
contains all the kinks of $g$.

\begin{remark}\label{rem:2.18}
For any $n\geq n_0-1$ and $k,\, 0\leq k< 2^n$ with the binary
expansion $k_n = \sum\limits_{i=0}^{n}x_i2^{n-i}$ we have that $
\delta_{n,k} = \Rot^{x_1}(\delta(g(I_{n,k}))). $
\end{remark}

\begin{proof}
Remark, that $g$ increase on $I_{n,k}$ if and only if $x_1 =0$.
Now our statement follows from the linearity of $g$ on $I_{n,k}$.
\end{proof}

\begin{remark}\label{rem:2.19}
For every $n\geq n_0$ we have that $\delta_{n+1,k_n}
=\Rot^{x_{n-n_0+2}}(
\delta_{n_0-1,\Rot_{n_0-1}^{x_{n-n_0+2}}(k_{n-n_0+2,n})}).$
\end{remark}

\begin{proof}
By Remark~\ref{rem:2.18},
$$\delta(I_{n+1,k}) = %
\Rot^{x_1}(\delta(g(I_{n+1,k})))
\stackrel{\text{Lem.~\ref{rem:2.17}}}{=}
\Rot^{x_1}(\delta(I_{n,\Rot_n^{x_1}(k_n)})) %
\stackrel{\text{Rem.~\ref{rem:2.18}}}{=} $$
$$
\Rot^{x_1}\circ
\Rot^{\Rot^{x_1}(x_2)}(\delta(I_{n-1,\Rot_{n-1}^{x_2}(k_{2,n})}))
\stackrel{\text{Rem.~\ref{rem:2.14}}}{=}
\Rot^{x_2}(\delta(I_{n-1,\Rot_{n-1}^{x_2}(k_{2,n})})).$$ It
follows now by induction that $$ \delta(I_{n+1,k}) = \Rot^{x_{i}}(
\delta(I_{n+1-i,\Rot_{n-i+1}^{x_{i}}(k_{i,n}}))
$$ for all $i\leq n$. Plug $i = n-n_0+2$ into the latter equality,
and we are done.
\end{proof}

\begin{remark}\label{rem:2.20}
For any $n\geq 1$ we have
$$\lI_{n+1,k_n} = \lI_{n,k_{n-1}}\cdot \Rot^{x_n +x_{n-n_0+1}}(
\delta_{\Rot_{n_0-1}^{x_{n-n_0+1}}(k_{n-n_0+1, n-1})}).
$$
\end{remark}

\begin{proof}
By Remark~\ref{rem:2.15}, $ \lI_{n+1,k_n} =\lI_{n,k_{n-1}}\cdot
\Rot^{x_n}(\delta_{n,k_{n-1}}), $ and the result follows from
Remark~\ref{rem:2.19}.
\end{proof}

\begin{notation}\label{not:2.21}
For any $k\geq 0$ denote $\delta_k = \delta(I_{n_0-1,k^*}),$ where
$k^*, 0\leq k^*< 2^{n_0-2}$ is such that $k -k^*$ is divisible by
$2^{n_0-2}$.
\end{notation}

The next fact follows from Remark~\ref{rem:2.20}.

\begin{remark}\label{rem:2.22}
For all $n\geq n_0-1$ and $k,\, 0\leq k\leq 2^{n-2}-1$ we have
$\delta_k = \delta_{n,k}$.
\end{remark}

\section{The main results}\label{sec:05}

\subsection{Technical computation}

Let a carcass map $g$ and a point
 $x\in [0, 1]$ be fixed till the end of
the section. As above, let $n_0$ be the minimal natural number
such that $g^{-n_0}(0)$ contains all the kinks of $g$.

\begin{notation}\label{not:3.2}
Denote $\mathcal{D} =\{ \delta_k,\, 0\leq k< 2^{n_0-2}\}$, where
$\delta_k$ are defined in Notation~\ref{not:2.21}. Write
$$ \delta_- =\min\limits_{\delta\in \mathcal{D}}\{ \delta,\, 1-\delta\}\text{ and }
\delta^+ =\max\limits_{\delta\in \mathcal{D}}\{ \delta,\,
1-\delta\}.
$$
\end{notation}

The following remark follows from Remark~\ref{rem:2.20}.

\begin{remark}\label{rem:3.3}
For any $n\geq n_0$ and any $p$ and $k$ such that $I_{n, k}\subset
I_{n_0-1,p}$ we have $$\lI_{n_0-1, p}\cdot
(\delta_-)^{n-n_0+1}\leq \lI_{n,k} \leq \lI_{n_0-1, p}\cdot
(\delta^+)^{n-n_0+1}.
$$
\end{remark}

\begin{remark}[Proof of Theorem~\ref{th:3}]
Notice, that Theorem~\ref{th:3} readily follows from
Remark~\ref{rem:3.3}.
\end{remark}

\begin{remark}\label{rem:3.5}
If $x_{n-n_0+1} = \ldots =x_n$, then $\lI_{n+1,k_n} =
\lI_{n,k_{n-1}}\cdot \delta_0.$
\end{remark}

\begin{proof}
This is the direct consequence of Remark~\ref{rem:2.20}.
\end{proof}

\begin{lemma}\label{lema:3.6}
Suppose that $n>n_0$ is such that $x_{n+1} = 1$. Denote $t\geq 0$
such that $x_{n+t+2}=1$ is the first one of the $g$-expansion of
$x$ after $x_{n+1}$. Then the following implications hold:

(i) For any $i,\, 0\leq i\leq n_0+1$ we have $ \lI_{n+1} \cdot
(\delta_-)^{i+1} \leq \lI_{n+i+2}^- \leq \lI_{n+1}\cdot
(\delta^+)^{i+1}. $

(ii) For every $i,\, n_0+2\leq i\leq t$ we have that
$$
\lI_{n+1}\cdot (\delta_-)^{n_0+2}\cdot \delta_0^{i-n_0-1} \leq
\lI_{n+i+2}^- \leq \lI_{n+1}\cdot (\delta^+)^{n_0+2}\cdot
\delta_0^{i-n_0-1}.
$$

(iii) For every $i,\, n_0\leq i\leq t-n_0+1$ we have that
$\lI_{p+i+1} = \lI_{n+n_0+1} \cdot \delta_0^{i-n_0}.$
\end{lemma}

\begin{proof}
1. Since $x_{i+1} =1$, then $\overline{I_{n+2}\cup I_{n+2}^-} =
\overline{I_{n+1}}$. Thus, $I_{p+i+2}^-\subset I_{p+1}$, and
part~(i) follows from Remark~\ref{rem:3.3}.

2. Remind that $I_{n+i+2}^- = I_{n+i+2,k_{n+i+1}-1}$. The last
binary digits of $k_{n+i+1}$ are $x_{n+i-n_0}$, $\ldots$,
$x_{n+i+1}$, which are $0$-s, since $n_0+2\leq i\leq t$. Thus,
Remark~\ref{rem:3.5} implies $\lI_{n+i+2}^- = \lI_{n+i+1}^- \cdot
\delta_0$ and, by induction, $\lI_{n+i+2}^- = \lI_{n+n_0+3}^-
\cdot \delta_0^{i-n_0-1}$. Now part~(ii) follows from part~(i).

3. The part (iii) follows from Remark~\ref{rem:3.5} due to the
construction of $t$.
\end{proof}

Write $g_1$ for $g$, and let $g_2$ be one more carcass map, which
will be fixed to the end of the section. Denote by $n_0$ the
minimal natural number such that $g_i^{-n_0}(0)$ contains all the
kinks of $g_i$, where $i\in \{1,\, 2\}$. For every $n\geq 1$ every
$k,\, 0\leq k\leq 2^{n-1}$ and $i\in \{1;\, 2\}$ define
$\mu_{n,k}(g_i)$, $I_{n,k}(g_i)$, $I_{n,k}^-(g_i)$,
$\lI_{n,k}(g_i)$, $\delta_{n,k}(g_i)$, $v(g_i)$,
$\mathcal{V}(g_i)$, $\delta_-(g_i)$ and $\delta^{+}(g_i)$ as in
Notations~\ref{not:2.2}, \ref{not:2.7} and~\ref{not:3.2}. Let $h$
be the conjugacy from $g_1$ to $g_2$.

The construction of Ulam's approximation implies the next:

\begin{remark}\label{rem:3.7}
If $x\notin g_1^{-n}(0)$ for some $n\geq 1$ then
$$
h_{n+1}'(x) =\frac{h(\iksp{n})
-h(\widehat{x}_n)}{\widehat{x}_n^{\, +} -\widehat{x}_n}.
$$
\end{remark}

The next fact follows from Lemma~\ref{lema:2.4}, part (iii) of
Lemma~\ref{lema:3.6}, and Remark~\ref{rem:3.7}.

\begin{remark}\label{rem:3.8}
For every $i\geq 1$ and $n\geq 1$, $$ \left(
\frac{\delta_-(g_2)}{\delta^+(g_1)}\right)^{i} \cdot
h_{n+1}'(x)\leq h_{n+i}'(x)\leq \left(
\frac{\delta^+(g_2)}{\delta_-(g_1)}\right)^{i} \cdot h_{n+1}'(x),
$$
whenever $x\notin g_1^{-n-i}(0)$.
\end{remark}

By Remarks~\ref{rem:3.5} and~\ref{rem:3.7} we have:

\begin{remark}\label{rem:3.9}
Suppose that $n>n_0$ is such that $x_{n+1} = 1$. Denote $t\geq 0$
such that $x_{n+t+2}=1$ is the first one of the $g$-expansion of
$x$ after $x_{n+1}$. Then for every $i,\, n_0\leq i\leq t-n_0+1$
we have
$$
h_{n+i}'(x) = h_{n+n_0}'(x)\cdot \left(
\frac{\delta_0(g_2)}{\delta_0(g_1)}\right)^{i-n_0}.
$$
\end{remark}

The next lemma will be an important step in the proof of
Theorem~\ref{th:4}.

\begin{lemma}\label{lema:3.10}
Suppose that $g_1^{n+1}(x)=0$ and $ s\in [\widehat{x}_n^{\, -},\,
\widehat{x}_n^{\, \mp})$ for $n>n_0$. Then
$$ \delta_-(g_2)\cdot \widehat{h}_{n+1}'(x-)\leq
\frac{h(\widehat{x}_n)-h(s)}{\widehat{x}_n-s}\leq
\widehat{h}_{n+1}'(x-)\cdot \frac{1}{\delta_-(g_1)}.
$$
\end{lemma}

\begin{proof}
Since $g_1^{n+1}(x)=0$, then $x = \widehat{x}_n$. Denote
$A(\widehat{x}_n^{\, -}, h(\widehat{x}_n^{\, -}))$, $S(s, h(s))$,
$X(x,\, h(x))$, $S_-(\widehat{x}_n^{\, -}, h(\widehat{x}_n^{\,
\mp}))$ and $S^+(\widehat{x}_n^{\, \mp}, h(\widehat{x}_n^{\, -}))$
(see Fig.~\ref{fig-1}a).

Also let $k_{S_-X}$ be the tangent of $S_-X$, let $k_{SX}$ be
$k_{SX}$ and let $k_{S^+X}$ be the tangent  of $S^+X$. Then
 $$ k_{S_-X}\leq k_{SX}\leq k_{S^+X},
$$ because
$ s\in [\widehat{x}_n^{\, -},\, \widehat{x}_n^{\, \mp})$ and $h$
increase.

By Remark~\ref{rem:3.3} and Lemma~\ref{lema:2.4}
$$
k_{S_-X}\geq \frac{(h(x)-h(\widehat{x}_n^{\, -}))\cdot
\delta_-(g_2)}{x- \widehat{x}_n^{\, -}}
$$ and $$
k_{S^+X} \leq \frac{h(x)-h(\widehat{x}_n^{\, -})}{(x -x_n^{\,
-})\cdot \delta_-(g_1)}.
$$ %
Now lemma follows from Remark~\ref{rem:3.7}.
\end{proof}

\begin{lemma}\label{lema:3.11}
Suppose that $x$ is $g_1$-infinite and $ s\in [\widehat{x}_n,\,
\widehat{x}_n^{\, +})$ for $n>n_0$. Assume that $x_{n+1} = 1$, and
let $t\geq 0$ be such that $x_{n+t+2}=1$ is the first one of the
$g$-expansion of $x$ after $x_{n+1}$. Then there exist $k_-$ and
$k^+$, independent on $x$, and $i\geq 1$ such that
$$
k_-\cdot h_{n+i}'(\widehat{x}_n-)  \leq
\frac{h(\widehat{x}_n)-h(s)}{\widehat{x}_n-s} \leq k^+\cdot
h_{n+i}'(\widehat{x}_n-)\ .
$$
\end{lemma}

\begin{proof}
Since $h$ increase, then for any $s_-,\, s^+,\, x_-$ and $x^+$
such that $$
\left\{ \begin{array}{l} s_-\leq s\leq s^+<x^-\leq x\leq x^+\\
s_-<s^+<x_-<x^+,
\end{array}\right.
$$ we have that
$$
k_{S_-X^+}\leq k_{SX}\leq k_{S^+X_-}, $$ %
where:

1. $k_{SX}$ is the tangent of the line, which connects points $(s,
h(s))$ and $(x, h(x))$;

2. $k_{S_-X^+}$ is the tangent of the line, which connects points
$(s_-, h(s^+))$ and $(x^+, h(x_-))$ and, finally

3. $k_{S^+X_-}$ is the tangent of the line, which connects points
$(s^+, h(s_-))$ and $(x_-, h(x^+))$.

By definitions, $$ k_{SX} =\frac{h(x)-h(s)}{x-s}, $$
$$
k_{S_-X^+} = \frac{h(x_-)-h(s^+)}{x^+-s_-}, $$ and
$$
k_{S^+X_-} = \frac{h(x^+)-h(s_-)}{x_--s^+}.
$$

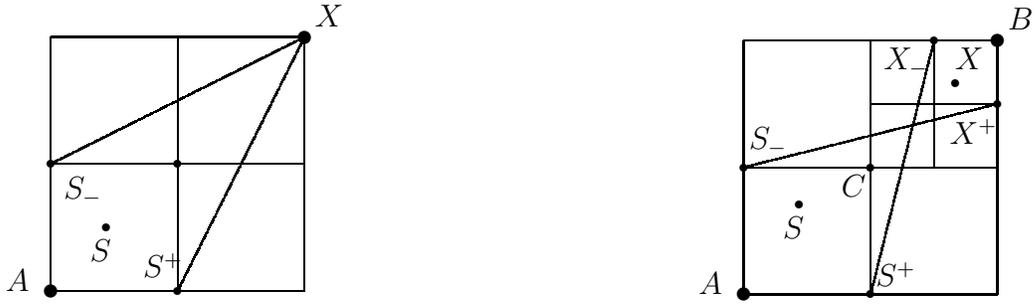
\begin{figure}[htbp]
\begin{minipage}[h]{0.45\linewidth}
\begin{center}
\begin{picture}(110,110)
\put(0,0){\line(0,1){96}} \put(0,0){\line(1,0){96}}

\put(0,48){\line(1,0){96}} 
\put(0,96){\line(1,0){96}}

\put(48,0){\line(0,1){96}} 
\put(96,0){\line(0,1){96}}

\put(0,0){\circle*{5}} \put(96,96){\circle*{5}} \put(-17,0){$A$}
\put(100,100){$X$}


\put(48,48){\circle*{3}} 


\put(21,24){\circle*{3}} \put(15,11){$S$}

\put(48,0){\circle*{3}} \put(35,5){$S^+$}

\put(0,48){\circle*{3}} \put(5,35){$S_-$}

\qbezier(0,48)(48,72)(96,96) \qbezier(48,0)(72,48)(96,96)
\end{picture}
\vskip 3mm \centerline{a) Construction of} \centerline{$S$, $S^+$,
$S_-$ and $X$}\end{center}
\end{minipage}
\hfill
\begin{minipage}[h]{0.45\linewidth}
\begin{center}
\begin{picture}(110,110)
\put(0,0){\line(0,1){96}} \put(0,0){\line(1,0){96}}

\put(0,48){\line(1,0){96}} \put(48,72){\line(1,0){48}}
\put(0,96){\line(1,0){96}}

\put(48,0){\line(0,1){96}} \put(72,48){\line(0,1){48}}
\put(96,0){\line(0,1){96}}

\put(0,0){\circle*{5}} \put(96,96){\circle*{5}} \put(-17,0){$A$}
\put(100,100){$B$}

\put(80,80){\circle*{3}} \put(80,85){$X$}

\put(48,48){\circle*{3}} \put(37,37){$C$}

\put(21,34){\circle*{3}} \put(15,21){$S$}

\qbezier(0,48)(48,60)(96,72) \qbezier(48,0)(60,48)(72,96)

\put(0,48){\circle*{3}} \put(48,0){\circle*{3}}
\put(72,96){\circle*{3}} \put(96,72){\circle*{3}}

\put(2,55){$S_-$}%
\put(50,3){$S^+$}

\put(53,85){$X_-$}%
\put(78,57){$X^+$}

\end{picture}
\vskip 3mm \centerline{b) Case $t =0$}\centerline{in
Lemma~\ref{lema:3.11}}\end{center}
\end{minipage}
\hfill \caption{Proof of Lemmas~\ref{lema:3.10}
and~\ref{lema:3.11}} \label{fig-1}
\end{figure}

If $s\in [\widehat{x}_{n+i+1}^{\, -}, \widehat{x}_{n+i+2}^{\, -})$
for some $i, 0\leq i< t$ then take
\begin{equation}\label{eq:3.1}
\begin{array}{ll}
s_- = \widehat{x}_{n+i+1}^{\, -},& s^+ =\widehat{x}_{n+i+2}^{\, -},\\
x_- = \widehat{x}_{n+1},& x^+ =\widehat{x}_{n+i+2}^{\, +}.
\end{array}
\end{equation}
And if $s\in [\widehat{x}_{n+t+1}^{\, -}, \widehat{x}_{n+1})$ then
take
\begin{equation}\label{eq:3.2}
\begin{array}{ll}
s_- = \widehat{x}_{n+t+1}^{\, -},& s^+ =\widehat{x}_{n+1},\\
x_- = \widehat{x}_{n+t+2},& x^+ =\widehat{x}_{n+t+1}^{\, +}.
\end{array}
\end{equation}

The case $t=0$ is presented at Figure~\ref{fig-1}b. The case $t=1$
is illustrated at Figure~\ref{fig-2}.

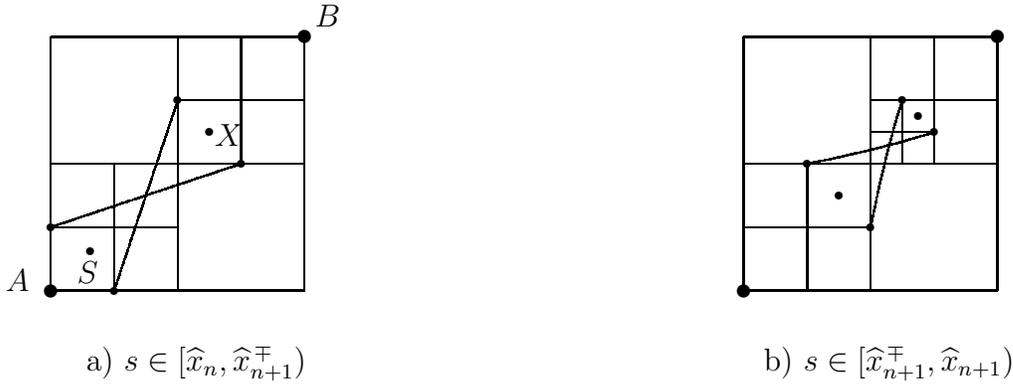
\begin{figure}[htbp]
\begin{minipage}[h]{0.45\linewidth}
\begin{center}
\begin{picture}(110,110)
\put(0,0){\line(0,1){96}} \put(0,0){\line(1,0){96}}

\put(0,48){\line(1,0){96}} \put(48,72){\line(1,0){48}}
\put(0,96){\line(1,0){96}}

\put(48,0){\line(0,1){96}} \put(72,48){\line(0,1){48}}
\put(96,0){\line(0,1){96}}

\put(0,0){\circle*{5}} \put(96,96){\circle*{5}} \put(-17,0){$A$}
\put(100,100){$B$}

\qbezier(24,0)(24,24)(24,48) \qbezier(0,24)(24,24)(48,24)

\put(60,60){\circle*{3}} \put(62,55){$X$}


\put(15,15){\circle*{3}} \put(10,3){$S$}


\qbezier(0,24)(36,36)(72,48) \qbezier(24,0)(36,36)(48,72)
\put(0,24){\circle*{3}} \put(24,0){\circle*{3}}
\put(48,72){\circle*{3}} \put(72,48){\circle*{3}}
\end{picture}
\vskip 3mm \centerline{a) $s\in [\widehat{x}_n,
\widehat{x}_{n+1}^{\, \mp})$}\end{center}
\end{minipage}
\hfill
\begin{minipage}[h]{0.45\linewidth}
\begin{center}
\begin{picture}(110,110)
\put(0,0){\line(0,1){96}} \put(0,0){\line(1,0){96}}

\put(0,48){\line(1,0){96}} \put(48,72){\line(1,0){48}}
\put(0,96){\line(1,0){96}}

\put(48,0){\line(0,1){96}} \put(72,48){\line(0,1){48}}
\put(96,0){\line(0,1){96}}

\put(0,0){\circle*{5}} \put(96,96){\circle*{5}}

\qbezier(24,0)(24,24)(24,48) \qbezier(0,24)(24,24)(48,24)

\put(66,66){\circle*{3}} 


\put(36,36){\circle*{3}} 

\qbezier(60,48)(60,60)(60,72) \qbezier(48,60)(60,60)(72,60)

\qbezier(24,48)(54,54)(72,60) \qbezier(48,24)(54,54)(60,72)
\put(24,48){\circle*{3}} \put(48,24){\circle*{3}}
\put(72,60){\circle*{3}} \put(60,72){\circle*{3}}
\end{picture}
\vskip 3mm \centerline{b) $s\in [\widehat{x}_{n+1}^{\, \mp},
\widehat{x}_{n+1})$}\end{center}
\end{minipage}
\hfill \caption{Case $t=1$ in Lemma~\ref{lema:3.10}} \label{fig-2}
\end{figure}

Suppose that $s\in [\widehat{x}_{n+i+1}^{\, -},
\widehat{x}_{n+i+2}^{\, -})$ for some $i,\, 0\leq i<t$, and that
the numbers $s_-,\, s^+,\, x_-$ and $x^+$ are defined
by~\eqref{eq:3.1}. In order to use Lemma~\ref{lema:3.6}, write
$$ x^+ -s_- =
(\widehat{x}^+_{n+i+2}- \widehat{x}_{n+1}) + (\widehat{x}_{n+1}
-\widehat{x}_{n+i+1}^{\, -})$$ and then we will calculate
restrictions of each summand.

For $0\leq i\leq n_0$ we have
\begin{equation}\label{eq:3.3}
x_- -s^+ = \widehat{x}_{n+1} - \widehat{x}_{n+i+2}^{\, -} =
\widehat{x}_n^+ -\widehat{x}_{n+i+2}^- \leq \lI_{n+1}(g_1).
\end{equation}
Next, by~(i) of Lemma~\ref{lema:3.6},
\begin{equation}\label{eq:3.4}
\widehat{x}_{n+1} -s_- = \lI_{n+i+2}^-(g_1)\geq \lI_{n+1}(g_1)
\cdot (\delta_-(g_1))^{n_0+1}.\end{equation}

If $n_0+1\leq i<t$, then by~(ii) of Lemma~\ref{lema:3.6},
\begin{equation}\label{eq:3.5}
x_- -s^+ \leq \lI_{n+1}(g_1) \cdot (\delta^+(g_1))^{n_0+2}\cdot
(\delta_0(g_1))^{i-n_0}
\end{equation} and
\begin{equation}\label{eq:3.6}
\lI_{n+1}(g_1)\cdot (\delta_-(g_1))^{n_0+2}\cdot
\delta_0^{i-n_0-1}(g_1) \leq \widehat{x}_{n+1} -s_-.
\end{equation}

For $i,\, 0\leq i\leq n_0-2$ it follows from Remark~\ref{rem:3.3}

\begin{equation}\label{eq:3.7}
\lI_{n+1}(g_1)\cdot (\delta_-(g_1))^{n_0}\leq x^+
-\widehat{x}_{n+1} \ .\end{equation}

If $n_0-1\leq i\leq t-n_0-1$ then it follows from~(iii) of
Lemma~\ref{lema:3.6} and Remark~\ref{rem:3.5} that
\begin{equation}\label{eq:3.8}
\lI_{n+1}(g_1)\cdot (\delta_-(g_1))^{n_0}\cdot
(\delta_0(g_1))^{i+2-n_0}\leq x^+ -\widehat{x}_{n+1} \
.\end{equation}

If $t-n_0\leq i<t$, then it follows from~(iii) of
Lemma~\ref{lema:3.6}, and Remark~\ref{rem:3.5} that
\begin{equation}\label{eq:3.9} \lI_{n+1}(g_1)\cdot
(\delta_0(g_1))^{t-2n_0+1}\cdot (\delta_-(g_1))^{2n_0} \ \ \ \leq
x^+ -\widehat{x}_{n+1}.
\end{equation}

If $0\leq i\leq n_0-2$ then if follows from~\eqref{eq:3.4}
and~\eqref{eq:3.7} that
\begin{equation}\label{eq:3.10}\lI_{n+1}(g_1)\cdot 2\cdot
(\delta_-(g_1))^{n_0+1}\leq x^+ -s_-.
\end{equation}

If $n_0-1\leq i\leq n_0$ then it follows from~\eqref{eq:3.4}
and~\eqref{eq:3.8} that
\begin{equation}\label{eq:3.11}\lI_{n+1}(g_1)\cdot 2\cdot
(\delta_-(g_1))^{n_0+1}\cdot \delta_0^{i+2-n_0}\leq x^+ -s_-.
\end{equation}

If $n_0+1\leq i\leq t-n_0-1$ then, by~\eqref{eq:3.6}
and~\eqref{eq:3.8} obtain
\begin{equation}\label{eq:3.12}
\lI_{n+1}(g_1)\cdot 2\cdot (\delta_-(g_1))^{n_0+2}\cdot
\delta_0^{i+2-n_0}\leq x^+ -s_-.
\end{equation}

If $t-n_0 \leq i\leq t-1$ then, by~\eqref{eq:3.6}
and~\eqref{eq:3.9} obtain \begin{equation}\label{eq:3.13}
\lI_{n+1}(g_1)\cdot 2 \cdot (\delta_-(g_1))^{2n_0} \cdot
(\delta_0(g_1))^{t-n_0} \leq x^+ -s_- \ .
\end{equation}

If $0\leq i\leq n_0-2$ then, by~\eqref{eq:3.3} and~\eqref{eq:3.10}

\begin{equation}\label{eq:3.14}
K_{S_-X^+} \geq \frac{\lI_{n+1}(g_2)\cdot 2\cdot
(\delta_-(g_2))^{n_0+1}}{\lI_{n+1}(g_1)}
\end{equation} and \begin{equation}\label{eq:3.15}
K_{S^+X_-} \leq \frac{\lI_{n+1}(g_2)}{\lI_{n+1}(g_1)\cdot 2\cdot
(\delta_-(g_1))^{n_0+1}}.
\end{equation}

If $n_0-1\leq i\leq n_0$ then by~\eqref{eq:3.3}
and~\eqref{eq:3.11}
\begin{equation}\label{eq:3.16}
K_{S_-X^+} \geq \frac{\lI_{n+1}(g_2)\cdot 2\cdot
(\delta_-(g_2))^{n_0+1}\cdot
(\delta_0(g_2))^{i+2-n_0}}{\lI_{n+1}(g_1)}
\end{equation} and \begin{equation}\label{eq:3.17}
K_{S^+X_-} \leq \frac{\lI_{n+1}(g_2)}{\lI_{n+1}(g_1)\cdot 2\cdot
(\delta_-(g_1))^{n_0+1}\cdot (\delta_0(g_1))^{i+2-n_0}}.
\end{equation}

If $n_0+1\leq i\leq t-n_0-1$ then by~\eqref{eq:3.5}
and~\eqref{eq:3.12} \begin{equation}\label{eq:3.18} K_{S_-X^+}
\geq \frac{\lI_{n+1}(g_2)\cdot 2\cdot (\delta_-(g_2))^{n_0+2}\cdot
(\delta_0(g_2))^{i+2-n_0}}{\lI_{n+1}(g_1)\cdot
(\delta^+(g_1))^{n_0+2}\cdot (\delta_0(g_1))^{i-n_0}}
\end{equation} and \begin{equation}\label{eq:3.19}
K_{S^+X_-} \leq \frac{\lI_{n+1}(g_2)\cdot
(\delta^+(g_2))^{n_0+2}\cdot
(\delta_0(g_2))^{i-n_0}}{\lI_{n+1}(g_1)\cdot 2\cdot
(\delta_-(g_1))^{n_0+2}\cdot (\delta_0(g_1))^{i+2-n_0}}.
\end{equation}

If $t-n_0 \leq i\leq t-1$ then by~\eqref{eq:3.5}
and~\eqref{eq:3.13}
\begin{equation}\label{eq:3.20}
K_{S_-X^+} \geq \frac{\lI_{n+1}(g_2)\cdot 2 \cdot
(\delta_-(g_2))^{2n_0} \cdot
(\delta_0(g_2))^{t-n_0}}{\lI_{n+1}(g_1)\cdot
(\delta^+(g_1))^{n_0+2}\cdot (\delta_0(g_1))^{i-n_0}}
\end{equation} and \begin{equation}\label{eq:3.21}
K_{S^+X_-} \leq \frac{\lI_{n+1}(g_2)\cdot
(\delta^+(g_2))^{n_0+2}\cdot
(\delta_0(g_2))^{i-n_0}}{\lI_{n+1}(g_1)\cdot 2 \cdot
(\delta_-(g_1))^{2n_0} \cdot (\delta_0(g_1))^{t-n_0}}.
\end{equation}
Due %
to Remark~\ref{rem:3.7}, it follows from~\eqref{eq:3.14},
\eqref{eq:3.15}, \eqref{eq:3.16} and~\eqref{eq:3.17} that for any
$i,\, 0\leq i\leq n_0$
$$
K_{S_-X^+}\geq h_{n+1}'(x)\cdot 2\cdot
(\delta_-(g_2))^{n_0+1}\cdot (\delta_0(g_2))^2
$$
and
$$
K_{S^+X_-}\leq \frac{h_{n+1}'(x)}{ 2\cdot
(\delta_-(g_1))^{n_0+1}\cdot (\delta_0(g_1))^2}.
$$

By Remark~\ref{rem:3.8}, \begin{equation}\label{eq:3.22}\left(
\frac{\delta_-(g_2)}{\delta^+(g_1)}\right)^{n_0} \cdot
h_{n+1}'(x)\leq h_{n+n_0}'(x)\leq \left(
\frac{\delta^+(g_2)}{\delta_-(g_1)}\right)^{n_0} \cdot
h_{n+1}'(x).\end{equation}

If $n_0+1\leq i\leq t-n_0-1$, then it follows from~\eqref{eq:3.18}
that
$$ K_{S_-X^+} \geq \frac{\lI_{n+1}(g_2)\cdot 2\cdot (\delta_-(g_2))^{n_0+2}\cdot
(\delta_0(g_2))^{i+2-n_0}}{\lI_{n+1}(g_1)\cdot
(\delta^+(g_1))^{n_0+2}\cdot
(\delta_0(g_1))^{i-n_0}}\stackrel{\text{by
Rem.~\ref{rem:3.7}}}{=}$$$$ = h_{n+1}'(x)\cdot \left(
\frac{\delta^+(g_2)}{\delta_-(g_1)}\right)^{n_0}\cdot \frac{2\cdot
(\delta_-(g_2))^{n_0+2}\cdot
(\delta_0(g_2))^{i+2-n_0}}{(\delta^+(g_1))^{n_0+2}\cdot
(\delta_0(g_1))^{i-n_0}}\cdot \left(
\frac{\delta_-(g_1)}{\delta^+(g_2)}\right)^{n_0}\stackrel{\text{by
Rem.~\ref{rem:3.9}}}{=}
$$$$
= h'_{n_0+n}(x)\cdot \left(
\frac{\delta_0(g_2)}{\delta_0(g_1)}\right)^{i-n_0}\cdot
\frac{2\cdot (\delta_-(g_2))^2\cdot
(\delta_0(g_2))^2}{(\delta^+(g_1))^2} \cdot \left(
\frac{\delta_-(g_1)\cdot \delta_-(g_2)}{\delta^+(g_1)\cdot
\delta^+(g_2)}\right)^{n_0}\stackrel{\text{by~\eqref{eq:3.22}}}{\geq}
$$$$
\geq h'_{n+i}(x)\cdot \frac{2\cdot (\delta_-(g_2))^2\cdot
(\delta_0(g_2))^2}{(\delta^+(g_1))^2} \cdot \left(
\frac{\delta_-(g_1)\cdot \delta_-(g_2)}{\delta^+(g_1)\cdot
\delta^+(g_2)}\right)^{n_0}.
$$ Analogously, $$K_{S^+X_-}
\stackrel{\text{by~\eqref{eq:3.19}}}{\leq}
 h_{n+1}'(x)\cdot \frac{ (\delta^+(g_2))^{n_0+2}\cdot
(\delta_0(g_2))^{i-n_0}}{ 2\cdot (\delta_-(g_1))^{n_0+2}\cdot
(\delta_0(g_1))^{i+2-n_0}}\stackrel{\text{by
Rem.~\ref{rem:3.7}}}{=}
$$$$ = h_{n+1}'(x)\cdot \left( \frac{\delta_-(g_2)}{\delta^+(g_1)}\right)^{n_0}
\cdot \frac{ (\delta^+(g_2))^{n_0+2}\cdot
(\delta_0(g_2))^{i-n_0}}{ 2\cdot (\delta_-(g_1))^{n_0+2}\cdot
(\delta_0(g_1))^{i+2-n_0}}\cdot \left(
\frac{\delta^+(g_1)}{\delta_-(g_2)}\right)^{n_0}\stackrel{\text{by
Rem.~\ref{rem:3.9}}}{=}$$
$$ =h_{n_0+n}'(x)\cdot \left( \frac{\delta_0(g_2)}{\delta_0(g_1)}\right)^{i-n_0}\cdot
\frac{ (\delta^+(g_2))^{n_0+2}}{ 2\cdot
(\delta_-(g_1))^{n_0+2}\cdot (\delta_0(g_1))^2}\cdot \left(
\frac{\delta^+(g_1)}{\delta_-(g_2)}\right)^{n_0}
\stackrel{\text{by~\eqref{eq:3.22}}}{\geq}
$$$$
\leq h'_{n+i}(x)\cdot \frac{ (\delta^+(g_2))^2}{2\cdot
(\delta_0(g_1))^2\cdot (\delta_-(g_2))^2} \cdot \left(
\frac{\delta^+(g_1)\cdot \delta^+(g_2)}{\delta_-(g_1)\cdot
\delta_-(g_2)}\right)^{n_0}.
$$

If $t-n_0\leq i\leq t-1$, then it follows from~\eqref{eq:3.20}
\eqref{eq:3.21}, \eqref{eq:3.22} and Remarks~\ref{rem:3.7}
and~\ref{rem:3.9} that
$$ K_{S_-X^+}
\geq h_{n+i}'(x)\cdot \frac{ 2 \cdot (\delta_-(g_2))^{n_0} }{
(\delta^+(g_1))^2}\cdot \left(  \frac{\delta_-(g_1)\cdot
\delta_-(g_2)\cdot \delta_0(g_2)}{\delta^+(g_1)\cdot
\delta^+(g_2)}\right)^{n_0}
$$
and
$$ K_{S^+X_-}
\leq h_{n+i}'(x)\cdot\frac{ (\delta^+(g_2))^2}{ 2 \cdot
(\delta_-(g_1))^{n_0}}\cdot \left( \frac{ \delta^+(g_2)\cdot
\delta^+(g_1)}{ \delta_-(g_1) \cdot \delta_-(g_2)\cdot
\delta_0(g_1)}\right)^{n_0}.
$$ This proves the lemma for the case~\eqref{eq:3.1}.

The case~\eqref{eq:3.2} can be considered analogously.
\end{proof}

In the same manner as in Lemmas~\ref{lema:3.10}
and~\ref{lema:3.11} we can prove the following lemma.

\begin{lemma}\label{lema:3.12}
Suppose that $x<1$. Then there exist $k_-$ and $k^+$, independent
on~$x$, such that for any $n\geq n_0$ and $ s\in (\widehat{x}_n,\,
\widehat{x}_n^{\, +}]$ there exits $i\geq 1$ such that
$$
k_-\cdot h_{n+i}'(\widehat{x}_n+)  \leq
\frac{h(\widehat{x}_n)-h(s)}{\widehat{x}_n-s} \leq k^+\cdot
h_{n+i}'(\widehat{x}_n+)\ .
$$
\end{lemma}

For every $n>n_0-1$, $k,\, 0\leq k<2^{n-1}$ and $x\in
I_{n,k}(g_1)$ denote by $\widehat{h}_{n,x}:\, [0, 1]\rightarrow
[0, 1]$ the piecewise linear map, such that all its kinks belong
to the set $A_{n,x} =g_1^{-n}(0)\cup \{x\}$ and
$\widehat{h}_{n,x}(t) = h(t)$ for all $t\in A_{n,x}$. Write $$
\Delta_L(I_{n,k}(g_1),t) =
\frac{\widehat{h}_{n,x}'(t-)}{h_n'(t)},\ \
\Delta_R(I_{n,k}(g_1),t) =
\frac{\widehat{h}_{n,x}'(t+)}{h_n'(t)},$$ and
$$\Delta(I_{n,k}(g_1),t) = \left\{
 \Delta_L(I_{n,k}(g_1),t); \Delta_R(I_{n,k}(g_1),t)\right\},
$$ see Fig.~\ref{fig-3}.

\begin{remark}\label{rem:3.13}
Notice that the following conditions are equivalent:

1. $\Delta_L(I_{n,k}(g_1), t) =1$;

2. $\Delta_R(I_{n,k}(g_1), t) =1$;

3. The point $(t, h(t))$ belongs to the graph of $h_n$.
\end{remark}

Since $g_1$ is linear on $I_{n,k}(g_1)$, and $g_2$ is linear on
$I_{n,k}(g_2)$, then $$ \Delta(I_{n,k}(g_1),t) =
\Delta(g_1(I_{n,k}(g_1)),g_1(t)),
$$ whence, by induction, \begin{equation}\label{eq:3.23}
\Delta(I_{n,k}(g_1),t) =
\Delta(g_1^{n-n_0+1}(I_{n,k}(g_1)),g_1^{n-n_0+1}(t))
\end{equation}

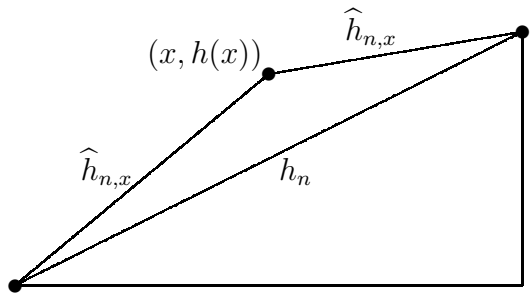
\begin{figure}[htbp]
\begin{center}
\begin{picture}(180,110)
 \put(0,0){\line(1,0){192}}


\put(192,0){\line(0,1){96}}

\put(0,0){\circle*{5}} \put(192,96){\circle*{5}}
\put(96,80){\circle*{5}} \put(50,84){$(x,h(x))$}
\put(100,40){$h_n$} \put(25,40){$\widehat{h}_{n,x}$}
\put(125,93){$\widehat{h}_{n,x}$}

\qbezier(0,0)(96,48)(192,96)

\qbezier(0,0)(48,40)(96,80) \qbezier(96,80)(144,88)(192,96)

\end{picture}
\end{center}
\caption{Graph of $\widehat{h}_{n,x}$ on $I_{n,k_{n-1}}$}
\label{fig-3}
\end{figure}

We are now ready to proof Theorem~\ref{th:4}.

\begin{proof}[Proof of Theorem~\ref{th:4}]
Part~(i) of Theorem~\ref{th:4} follows from Lemmas~\ref{lema:3.11}
and~\ref{lema:3.12}.

We will proof part (ii). Let $x$ has $g_1$-expansion $(x_i)_{i\geq
0}$ and $(k_n)_{n\geq 1}$ be a maximal number, such that $x\in
\overline{I_{n+1,k_n}}$ for all $n\geq 1$.

Suppose that $x\notin g_1^{-\infty}(0)$. If the derivative $h'(x)$
exists, then all $h_n'(x)$, $\widehat{h}_{n,x}'(x-)$ and
$\widehat{h}_{n,x}'(x-)$ tend to $h'(x)$. Moreover, if $h'(x)$ is
finite, then $\Delta_L(g_1(I_{n,k}(g_1)),g_1(x))$ and
$\Delta_R(g_1(I_{n,k}(g_1)),g_1(x))$ tend to $1$. Now the theorem
follows from Remark~\ref{rem:3.13} and~\eqref{eq:3.23}.

If $x=\mu_{n,k}(g_1)$ for some $k\in \{0,\ldots, 2^{n-1}-1\}$,
then existence and finiteness of the derivative $h'(x)$ implies
that $\lim\limits_{n \rightarrow \infty}\sup\limits_{t\in
I_{n+1,k_n}}\Delta_L(g_1(I_{n,k}(g_1)),g_1(t)) =1$ and, again we
are done by Remark~\ref{rem:3.13} and~\eqref{eq:3.23}.

The last case, $x =1$, follows from Remark~\ref{rem:3.13}
and~\eqref{eq:3.23}, if notice that $$\lim\limits_{n \rightarrow
\infty}\sup\limits_{t\in
I_{n+1,k_n}}\Delta_R(g_1(I_{n,k}(g_1)),g_1(t)) =1.$$
\end{proof}

We shall now prove Theorem~\ref{th:5}. This proof will be
constructive, i.e. we will give an explicit example of a map,
whose existence is mentioned in the theorem.

\begin{proof}[Theorem~\ref{th:5}]
The proof will be constructive. Suppose that $g$ is a piecewise
linear unimodal map, whose graph extends linearly $(0, 0)$ to
$(1/5, 1/2)$ to $(1/2, 1)$ to $(1, 0)$.

By construction, this $g$ is a carcass map. By Theorem~\ref{th:3},
the map $g$ is topologically conjugated to the tent map.

Denote $x^*$ the positive fixed point of $g$. Let $(k_n)_{n\geq
0}$ be such that $x^*\in I_{n+1,k_n}(g)$ for all $n\geq 0$. Since
$x^* = \lim\limits_{n\rightarrow \infty}g_r^{-n}(0)$, then
$I_{n+1,k_n} = g_r^{-n}([0, 1])$, whence
\begin{equation}\label{eq:3.24}
\left\{\begin{array}{l}
\mu_{n+1,k_n}(g) = \mu_{n+1,k_n}(f)\\
\mu_{n+1,k_n+1}(g) = \mu_{n+1,k_n+1}(f)
\end{array}\right.
\end{equation} for %
all $n\geq 1$. Now~\eqref{eq:3.24} means that $h_n'(x^*)=1$ for
all $n\geq 1$.

From another hand, it immediately follows from the equality $g(0)
= (h\circ f\circ h^{-1})(0)$ that $g'(0) =2$, whenever $h$ is
piecewise linear. The latter fact gives the contradiction, which
finishes the proof.
\end{proof}

\setlength{\unitlength}{1pt}

\pagestyle{empty}
\bibliography{Ds-Bib}{}
\bibliographystyle{makar}


\end{document}